\documentclass[reqno]{amsart}
\usepackage{amsmath}
\usepackage{dsfont}
\usepackage{amsfonts}
\usepackage{amssymb}
\usepackage{aliascnt}
\usepackage{graphicx}
\usepackage{mathrsfs}
\usepackage{enumerate}
\usepackage[bookmarks=true,pdfstartview=FitH, pdfborder={0 0 0}, colorlinks=true,citecolor=blue, linkcolor=blue]{hyperref}
\usepackage{tikz}
\newtheorem{theorem}{Theorem}[section]
\newaliascnt{conj}{theorem}
\newaliascnt{cor}{theorem}
\newaliascnt{lemma}{theorem}
\newaliascnt{fact}{theorem}
\newaliascnt{claim}{theorem}
\newaliascnt{prop}{theorem}
\newaliascnt{definition}{theorem}

\newtheorem{conj}[conj]{Conjecture}
\newtheorem{cor}[cor]{Corollary}
\newtheorem{lemma}[lemma]{Lemma}

\newtheorem{prop}[prop]{Proposition}

\newtheorem{definition}[definition]{Definition}

\aliascntresetthe{conj} \aliascntresetthe{cor}
\aliascntresetthe{lemma} \aliascntresetthe{fact}
\aliascntresetthe{claim} \aliascntresetthe{prop}
\aliascntresetthe{definition}

\theoremstyle{definition}
\newaliascnt{example}{theorem}
\newtheorem{example}[example]{Example}
\aliascntresetthe{example}

\theoremstyle{remark}
\newaliascnt{rmk}{theorem}
\newtheorem{remark}[rmk]{Remark}
\aliascntresetthe{rmk} 

\def\sek~{\S{}}

\numberwithin{equation}{section}

\newcommand{\diam}{\operatorname{diam}}

\newcommand{\mdk}{\operatorname{md_{\kappa}}}
\newcommand{\rad}{\operatorname{Rad}}
\newcommand{\inrad}{\operatorname{InRad}}
\newcommand{\vol}{\operatorname{Vol}}

\newcommand{\gexp}{\operatorname{gexp}}
\DeclareMathOperator{\arccot}{\operatorname{arccot}}
\newcommand{\alex}{\operatorname{Alex}}	
\newcommand{\distbx}{\operatorname{dist_{\partial X}}}	
\newcommand{\RR}{\mathds{R}}

\newcommand{\ZZ}{\mathds{Z}}

\newcommand{\BA}{\operatorname{BA}}
\renewcommand{\SS}{\mathbf{S}}

\newcommand{\cD}{{\mathcal D}}

\newcommand{\bx}{\partial X}
\newcommand*\circled[1]{\tikz[baseline=(char.base)]{\node[shape=circle,draw,inner sep=2pt] (char) {#1};}}

\begin{document}
\title{Radius Estimates for Alexandrov Space with Boundary}
\author{Jian Ge}
\address[Ge]{Beijing International center for Mathematical Research, Peking University. Beijing 100871, China}
\email{jge@math.pku.edu.cn}

\author{Ronggang Li}
\address[Li]{School of Mathematical Science, Peking University. Beijing 100871, China}
\email{lrg@pku.edu.cn}


\begin{abstract}
In this note, we study the radius of positively curved or non-negatively curved Alexandrov space with strictly convex boundary, with convexity measured by the Base-Angle defined by Alexander and Bishop. We also estimate the volume of the boundary of non-negatively curved spaces as well as the rigidity case, which can be thought as a non-negatively curved version of a recent result of Grove-Petersen.
\end{abstract}
\maketitle
%------------------ Section 0----------------
\section{Introduction}
Let $M^{n}$ be a closed $n$-dimensional Riemannian manifold with Ricci curvature bound from below by $(n-1)$, then by the classical Bonnet-Myers theorem the diameter of $M$ has the upper bound: $\diam(M)\le \pi$. Let $X\in \alex^{n}(1)$, i.e. an $n$-dimensional Alexandrov space with curvature bounded from below by $1$, we have the same diameter estimate $\diam(X)\le \pi$, by \cite{BGP1992}. The positive lower bound of the curvature is crucial here, since for any $X\in \alex^{n-1}(0)$, the cylinder $X\times \RR\in \alex^{n}(0)$ has infinite diameter. On the other hand, if  the Ricci curvature of $M^{n}$ is nonnegative, and $\partial M$ is non-empty with mean curvature satisfies $H\ge (n-1)h>0$, we can still estimate the inner radius of $M$, i.e. the largest radius of a metric ball inscribed inside the manifold: $\inrad(M)\le 1/h$. cf. \cite{Li2014}. Cf. also \cite{Ge2015} for a unified treatment for all lower curvature bounds. In this case, one cannot estimate the diameter as the solid cylinder with cross section a unit disc: $D^{n-1}\times \RR$ shows. For Alexandrov spaces, one expects that a similar estimate holds. First, the mean curvature assumption in \cite{Li2014} needs to be replaced by something meaningful for non-smooth spaces. This has been done by Alexander-Bishop in \cite{AB2010}, where the authors defined a  function called Base-Angle at each foot point.
\begin{definition}[\cite{AB2010}]\label{def:BA}
Let $X$ be an $n$-dimensional Alexandrov space with non-empty boundary $\partial X$. For $x\in \partial X$, the \it{base angle at} $p$ of a chord $\gamma$ of $V$ at an endpoint $p$ is the angle formed by the direction of $\gamma$ and $\partial(\Sigma_x(X))$, where $\Sigma_{x}(X)$ is the space of directions at $x$ of $X$. We call the boundary $\partial X$ has extrinsic curvature $\ge A$  in the base-angle sense at $x$ or $\BA(x)\ge A$, if the base angle $\alpha$ at $x$ of a chord of length $r$ from $x$ satisfies
$$
\liminf_{r\to 0}\frac{2\alpha}{r}\ge A.
$$
\end{definition}
It can be verified that if $X$ is a Riemannian manifold with smooth boundary, a Base-Angle lower bound is equivalent to a lower bound on the principal curvatures the boundary. We will call the boundary $\partial X$ is $A$-convex, if the base-angle $\BA(x)\ge A$ at each foot point, which will be written as $\BA(\partial X)\ge A$. Recall that a point $x\in \partial X$ is called a {\it foot point}, if there exists $y\in X\setminus \partial X$ such that
$$
\rho(y):=|y, \partial X| =|y, x|.
$$
We use $|A, B|$ to denote the distance between subsets $A$ and $B$ in $X$. In \cite{AB2010} it is then proved, among other things, that the inner radius of $X\in \alex^{n}(\kappa)$ with $A$-convex boundary $\partial X$ satisfies the expected estimate, see \autoref{cor:inner}.

In this note, we are interested in the radius estimate for $X\in\alex^{n}(\kappa)$ with $A$-convex boundary $\partial X$. Recall the radius of $X$ at $p$ is defined by
$$
\rad_{p}(X)=\sup\{|p, x|\ |\ x\in X\},
$$
and the radius of $X$ is defined by
$$
\rad(X)=\inf_{p\in X}\rad_{p}(X).
$$

Now we state our main theorems
\begin{theorem}\label{thm:k=0}
Let $X\in\alex^{n} (0)$, with $\BA(\partial X)\ge A> 0$. We have:
$$
\rad(X)\leq \frac{1}{A},
$$
with equality holds if and only if $X$ is isometric to the warped product $[0, A]\times_{t}\partial X$.
\end{theorem}
	
\begin{theorem}\label{thm:k=1}
Let $X\in\alex^{n} (1)$, with $\BA(\partial X)\ge A\geq 0$. We have:
$$
\rad(X)\leq\arccot (A),
$$
with equality holds if and only if $X$ is isometric to the warped product $[0, \arccot(A)]\times_{\sin(t)}\partial X$.
\end{theorem}

\begin{remark}
As one can easily see, our upper bound of the radius is the same as the upper bound of inner-radius proved in \autoref{cor:inner} by Alexander-Bishop, but our theorem does not imply their estimate since we use the inner radius estimate in our proof of radius estimate. On the other hand, our result gives shaper estimates of inner-radius, in fact, we insert more terms between the inner-radius and Alexander-Bishop's upper bound. See \autoref{thm:b:k=1} and \autoref{thm:b:k=0} for details.
\end{remark}

Let $X\in\alex^{n}(\kappa)$ with nonempty boundary $\partial X$, the \emph{Boundary Conjecture} says that $\partial X$ equipped with the induced path metric is again an Alexandrov space with the same lower curvature bound $\kappa$. In particular, if $\kappa=1$, we expect $\partial X$ has lower curvature bound $1$, thus it would follows from the Boundary Conjecture that $\diam(\partial X)\le \pi$ and $\vol(\partial X)\le \vol(\SS^{n-1})$, where $\SS^{n-1}$ denotes the unit $(n-1)$-sphere. The volume upper bound of $\partial X$ was called Lytchak's Problem in \cite{Pet2007}, and Petrunin proved it using gradient exponential map. The rigidity result is proved only recently by Grove-Petersen \cite{GP2018}. In the \cite{Ge2018}, the first author estimates the volume of Alexandrov space with fixed boundary, where we could think of the convexity of the boundary as positive curvatures. As the classical Gauss equation relates the intrinsic curvature of submanifold and ambient space via the second fundament form. So we propose the following Boundary Conjecture for Alexandrov spaces with curved boundary:
\begin{conj}
Let $X\in \alex^{n}(0)$ and $\BA(\partial X)\ge 1$, then $\partial X\in \alex^{n-1}(1)$.
\end{conj}

Our next theorem gives an evidence of this conjecture. Namely we get a solution to the Lytchak's Problem for the non-negatively curved Alexandrov space with $1$-convex boundary, as well as a rigidity result parallel to the one in \cite{GP2018}:
\begin{theorem}\label{thm:fill01}
Let $X\in \alex^{n}(0)$ with $\partial X\neq\emptyset$. Suppose $\BA(\partial X)\ge 1$. Then
$$
\vol_{n-1}(\bx)\leq \vol_{n-1}\big(\SS^{n-1}(1)\big).
$$
Moreover, if $\partial X$ is intrinsically isometric to $\SS^{n-1}$, then $X$ is isometric to the unit disk in $\RR^{n}$.
\end{theorem}

Note that in the classical positive mass theorem implies that the Euclidean $\RR^{n}$ admits not compact perturbation while keeping lower scalar curvature bound $0$. On the hand, the boundary hypersurface is assumed to be smooth or with a restricted type of singularity, cf. \cite{ST2002, ST2018} . Our approach to this problem uses no assumption on the smoothness of the boundary at all. However, we required a much strong curvature condition.\\

Acknowledgment: We would like to thank Stephanie Alexander and Yuguang Shi for their interest in our work and helpful discussions.

\section{Proofs of the Radius Estimates}
One key ingredient of our proof is the following concavity estimates of the distance function $\rho(x)=|x, \partial X|$:
\begin{theorem}[\cite{AB2010}]\label{thm:AB}
Let $X\in \alex^{n}(\kappa)$ and $\BA(\partial X)\ge A$. Let
$$\cD=R(\kappa, A)-\distbx$$
where $R(\kappa, A)$ is the radius of the circle with geodesic curvature equals to $A$ in the $2$-dimensional space form of curvature $\kappa$. If $\kappa>-A^2$, then $\cD$ is nonnegative, and the function $f=\mdk(\cD)$ satisfies
$$
f''+\kappa f \geq 1
$$
where $\mdk(t)=\int_{0}^{t} \frac{1}{\sqrt{\kappa}}\sin(\sqrt{\kappa}s) ds$.
\end{theorem}
	
The non-negativity of $\cD$ implies that the \emph{inner radius} estimate of $X$, i.e.
\begin{cor}[\cite{AB2010}, Cor. 1.9]\label{cor:inner}
Let $X$ and $R$ be as above, then the inner radius of $X$ satisfies
$$a:=\max_{x\in X}\rho(x)\le R(\kappa, A).$$
In particular, $a\le \frac1A$ for $\kappa=0$ and $a\le \arccot(A)$ for $\kappa=1$. Moreover, in the case $\kappa=0, A>0$ and $\kappa=1, A\ge 0$, there is a unique point $s\in X$ realized the maximum of $\rho$, which is called the \emph{soul} of $X$.	
\end{cor}
	
First, we need characterize the set of points with maximal distance to the soul $s\in X$.
\begin{lemma}\label{lem:rad}
Let $X\in \alex^{n}(\kappa)$ with $\kappa\ge 0$ and $\partial X\ne \varnothing$. Let $s$ be the soul of $X$. Then
$$
\rad_{s}(X)= \sup_{x\in \partial X}|s, x|=:b.
$$
\end{lemma}
\begin{proof}
For any $y$ in the interior of $X$, let $q\in \bx$ be a foot point such that $|y, \partial X| =|y, q|=:\beta$. Let $\gamma(t):[0,\beta]\rightarrow $ be the unit-speed geodesic from $y$ to $q$. We have:
$$
|\Uparrow_y^s,\gamma^{+}(0)|\geq \frac{\pi}{2},
$$
since otherwise there would exist a geodesic $\alpha$ from $y$ to $s$ with $\alpha(0)=y$ and $|\alpha^+(0),\gamma^+(0)|< \frac{\pi}{2}$, then by the first variation formula
$$
\rho(\alpha(\epsilon))\leq|\alpha(\epsilon)q|<|yq|=\rho(y)
$$
for $\epsilon$ small. Here the set $\Uparrow_{y}^{s}$ consists of initial directions of all the unit speed geodesics from $y$ to $s$. 

On the other hand, $\rho$ is a concave function with the maximum achieved at $s$, it follows that $\rho(\alpha(\cdot))$ is monotone, therefore $\rho(s)<\rho(y)$. Hence a contradiction.
		
Since $\forall t\in[0,\beta]$, $q$ is the foot point achieves the distance form $\gamma(t)$ to $\partial X$, we have:
$$
|\Uparrow_{\gamma(t)}^s,\gamma'(t)|\geq\frac{\pi}{2}.
$$
by replacing the $y$ above by $\gamma(t), t\in[0,\beta]$. Therefore the first variation formula tells $|\gamma(t),s|$ is increasing along $\gamma(t)$. It follows that
$$|qs| \geq |ys|.$$
Therefore the conclusion holds.
\end{proof}

\begin{remark}
It can be showed that $\rad_{s}(X)$ can only be achieved by geodesics from $s$ to some points on $\partial X$. In fact, if $|\Uparrow_{\gamma(t)}^s, \gamma'(t)|>\frac{\pi}{2}$ for some $t\in[0,\beta]$, then we have the strict inequality $|q, s|>|y, s|$. Therefore if there were an interior point $y\notin\partial X$ satisfies $|s, y|=\rad_{s}(X)$, it follows that $|\Uparrow_{\gamma(t)}^s,\gamma'(t)|\equiv\frac{\pi}{2}$ for any $t\in[0,\beta]$. In particular, the equality holds at $q=\gamma(\beta)$, therefore $\uparrow_q^s\in\partial\Sigma_{s}X$. By the convexity of $\partial X$, we have $s\in\partial X$. Hence a contradiction.
\end{remark}
	
The following elementary comparison result for ODEs is needed.
\begin{lemma}\label{lem:ode}
For any $\kappa\geq 0$, let $f$ and $\tilde{f}$ be real functions on $[0,\infty)$ satisfying:
\begin{align*}
f''+\kappa f&\geq 1 \\
\tilde{f}''+\kappa \tilde{f}&=1
\end{align*}	
respectively, while $0\leq f(0)<\frac{1}{\kappa}$ (if $\kappa=0$, define $\frac{1}{\kappa}$ as $\infty$), $f(0)=\tilde{f}(0)$, $f'(0)=\tilde{f}'(0)$. Then
$$
f(t)\geq \tilde{f}(t) \qquad\forall t \in [0,t_{0}]
$$
where $t_0$ is the first zero of $\frac{1}{\kappa}-\widetilde{f}$.
\end{lemma}
\begin{proof}
For the case $\kappa>0$, let $w=\frac{1}{\kappa}-f$ and $\tilde{w}=\frac{1}{\kappa}-\tilde{f}$. Then $w(0)=\tilde{w}(0)>0$, and the ordinary functions of $f$ and $\tilde{f}$ makes
\begin{align*}
w''+\kappa w&\leq 0\\
\tilde{w}''+\kappa \tilde{w} &=0.
\end{align*}
Then
\begin{align*}
&\qquad w''(t)\widetilde{w}(t)-\widetilde{w}''(t)w(t)\leq 0  \\
&\Longleftrightarrow \big(w'(t)\widetilde{w}(t)-\widetilde{w}'(t)w(t)\big)'\leq 0  \\
&\Longrightarrow\big(w'(t)\widetilde{w}(t)-\widetilde{w}'(t)w(t)\big) \leq \big(w'(0)\widetilde{w}(0)-\widetilde{w}'(0)w(0)\big)=0 \\
&\Longrightarrow (\frac{\widetilde{w}}{w})'(t)\geq 0 \qquad \text{whenever $w(t)>0$}	\\
&\quad\; \; \; \; (\frac{w}{\widetilde{w}})'(t)\leq 0 \qquad \text{whenever $\widetilde{w}(t)>0$}	\\
&\Longrightarrow w(t)\leq\widetilde{w}(t) \qquad \text{whenever $\widetilde{w}(t)\geq 0$}\\
&\Longleftrightarrow \frac{1}{\kappa}-f(t)\leq\frac{1}{\kappa}-\widetilde{f}(t) \qquad \text{whenever $\frac{1}{\kappa}-\widetilde{f}(t)\geq 0$}\\
&\Longleftrightarrow f(t)\geq\widetilde{f}(t) \qquad \text{whenever $\frac{1}{\kappa}-\widetilde{f}(t)\geq 0$}.\\
\end{align*}

In the case that $\kappa=0$, it is easy to see $f''-\tilde{f}''\geq 0$, thus $f'-\tilde{f}'\geq 0$ by $f'(0)=\tilde{f}'(0)$, and then, $f(t)\geq\tilde{f}(t)$ follows from $f(0)=\tilde{f}(0)$.
\end{proof}	
	
The \autoref{thm:k=1} and \autoref{thm:k=0} are in fact easy corollaries of the following theorems, where we insert one more term between the inner radius estimates of \autoref{cor:inner}.  As we can see easily
$$
a\le \rad(x)\le b,
$$
Recall that $a=\max_{x\in X} \rho(x)$ and $b=\rad_{s}(X)$.
We have:
\begin{theorem}\label{thm:b:k=1}
If $X\in\alex^{n} (1)$ and $\BA(\partial X)\ge A\geq 0$. Then
$$
\rad (X)\leq b\le \arccos \left (\frac{A}{A\cos a+\sin a}\right ).
$$
\end{theorem}
\begin{proof}
Set $\ell =R(1,A)=\arccot (A)$. Let $\gamma(t)$ be a geodesic of length $b$ with $\gamma(0)=s$ and $\gamma(b)\in \partial X$. Therefore $\distbx(\gamma(0))=a\leq \ell$.
Let
$$h(t)=\rho(\gamma(t)),$$
%$$w(t)=\cos(\ell-h(t)).$$
Then $h$ satisfies
$$
h(0)=a,\ h(b)=0, \ h'(0)=-\cos\alpha_{0},
$$	
where $\alpha_{0}=|\gamma'(0),\Uparrow_{s}^{\bx}|$. Since $s$ is the critical point for the distance function $\rho$, we have $\alpha_{0}\leq \frac{\pi}{2}.$  Define
\begin{align*}
f(t)=1-\cos(\ell-h(t)) \qquad 	t\in[0,b] .
\end{align*}
Since we are working for the case $\kappa=1$, $\mdk(x)=1-\cos x$. Therefore the function $f$ satisfies the following differential inequality:
\begin{align*}
&\qquad	                 f''+f\geq 1 \\
%&\Longleftrightarrow \big(\cos(\ell-h(t))\big)''+\cos(\ell-h(t))\leq 0\\
%&\Longleftrightarrow  w''(t)+w(t)\leq 0
\end{align*}
Let
%$$
%\widetilde{w}(t)=\cos(\ell-a)\cos t-\sin(\ell-a)\sin t\cos\alpha_0,
%$$
$$
\widetilde{f}(t)=1-\cos(\ell-a)\cos t+\sin(\ell-a)\sin t\cos\alpha_0.
$$
then one verifies easily:
\begin{align*}
%\widetilde{w}(0)&=w(0)=\cos(\ell-a) \\
\widetilde{f}(0)&=1-\cos(\ell-a)=f(0) \\
%\widetilde{w}'(0)&=w'(0)=-\sin(\ell-a)\cos\alpha_{0} \\
\widetilde{f}'(0)&=\sin(\ell-a)\cos\alpha_{0}=f'(0)
\end{align*}
and
%$$
%\widetilde{w}''(t)+\widetilde{w}(t)=0
%$$
$$
\widetilde{f}''(t)+\widetilde{f}(t)=1
$$
that follows:
$$f(t)\geq\widetilde{f}(t)$$
%thus
%$$w(t)\leq\widetilde{w}(t)$$
for $t\le t_{0}$, where the $t_{0}>0$ is the first zero of $1-\tilde f$ by \autoref{lem:ode}.
Especially when $t=b\leq t_0$, we have
\begin{align*}
&\qquad\cos \ell\leq \cos(\ell-a)\cos b-\sin(\ell-a)\sin b\cos\alpha_{0} \\
&\Longrightarrow\cos \ell\leq\cos(\ell-a)\cos b \\
&\Longleftrightarrow\frac{\cos \ell}{\cos(\ell-a)}\leq \cos b \\
&\Longleftrightarrow \frac{A}{A\cos a+\sin a}\leq \cos b \\
&\Longleftrightarrow b\leq \arccos \left( \frac{A}{A\cos a+\sin a} \right)
\end{align*}
\end{proof}
	
\begin{proof}[Proof of \autoref{thm:k=1}]
One observe that
$$\frac{A}{A\cos a+\sin a}\geq \frac{A}{\sqrt{1+A^2}},$$
since $a\leq l\leq \frac{\pi}{2}$.
We have:
$$
b\leq \arccos \left( \frac{A}{A\cos a+\sin a} \right)\le \arccos \frac{A}{\sqrt{1+A^2}}
$$
that is $b\leq \arccot(A)$.
\end{proof}
	
Now we move to the discussion on the case $\kappa =0$.
\begin{theorem}\label{thm:b:k=0}
Let $X\in\alex^{n} (0)$ and $\BA(\partial X)\ge A> 0$. Let $a$ be the inner radius of $X$. We have:
$$
\rad(X)\leq b\le \sqrt{2\frac{a}{A}-a^2}
$$
\end{theorem}
\begin{proof}
Let $\ell=R(0,A)=\frac{1}{A}$. Suppose $\gamma(t)$ is a geodesic of length $b$, with $\gamma(0)=s$ and $\gamma(b)\in\bx$. Therefore $\distbx(\gamma(0))=a\leq \frac{1}{A}$ by \autoref{cor:inner}.
Let
$$
h(t)=\distbx(\gamma(t)),
$$
then
$$
h(0)=a, \  h(b)=0, \ -h'(0)=\cos\alpha_{0}
$$	
where $\alpha_{0}=|\gamma'(0),\Uparrow_{s}^{\bx}|$. Since $s$ is the critical point for $\rho$, we know $\alpha_{0}\leq \frac{\pi}{2}.$
In this case,
$$
\mdk(x)=\frac{x^2}{2},
$$
thus
\begin{align*}
f(t)=\frac{(\frac{1}{A}-h(t))^2}{2}, \qquad 	0\leq h(t)\leq\frac{1}{A}
\end{align*}
satisfying
$$
f''\geq 1.
$$
Let
$$
\widetilde{f}(t)=
\frac{t^2+2(\frac{1}{A}-a)\cos\alpha_{0} t+(\frac{1}{A}-a)^2}{2}
$$
then
\begin{align*}
\widetilde{f}(0)&=f(0)=\frac{(\frac{1}{A}-a)^2}{2} \\
\widetilde{f}'(0)&=f'(0)=(\frac{1}{A}-a)\cos\alpha_{0}
\end{align*}
and
$$\widetilde{f}''(t)=1$$
It follows that
$$f(t)\geq\widetilde{f}(t).$$
Therefore, when $t=b$, we have
\begin{align*}
&\qquad b^2+2(\frac{1}{A}-a)\cos\alpha_{0} b+(a^2-2\frac{a}{A})\leq 0\\
&\Longrightarrow b\leq \sqrt{2\frac{a}{A}-a^2}
\end{align*}
\end{proof}
	
\begin{proof}[Proof of \autoref{thm:k=0}]
By \autoref{thm:b:k=0} and \autoref{cor:inner}:
$a\leq \frac1A,$
we have:
$$
b\le \sqrt{2\frac{a}{A}-a^2} \leq \frac{1}{A}.
$$
Thus the conclusion follows.
\end{proof}
	
\section{Discussion of the Equality Cases}
In this section we discuss various equality case in the estimates below. Recall that the inner radius $a:=\max_{x\in X}\rho(x)$ and $b:=\max_{x\in \partial X}|x, s|$. By the previous theorems we have for the case $\kappa=0, A>0$:
\begin{equation}\label{eq:equality:0}
a\le \rad(X)\le b\le \sqrt{2\frac aA-a^{2}}\le \frac 1A;
\end{equation}
and for the case $\kappa=1, A>0$:
\begin{equation}\label{eq:equality:1}
a \le \rad(X)\le b\le \arccos\left( \frac{A}{A\cos a+\sin a}\right)\le \arccot(A).
\end{equation}
For simplicity, we will refer the terms in the \eqref{eq:equality:0} and \eqref{eq:equality:1} as $\circled{1}$ to $\circled{5}$ from the left to the right.
\begin{prop}[$\circled{1}=\circled{5}$, \cite{AB2010}]\label{prop:rigidty15}
The equality $a=\frac1A$ in \eqref{eq:equality:0} (resp. $a=\arccot(A)$ in \eqref{eq:equality:1}) implies that the space $X$ is isometric to the cone $[0, a]\times_{t} \partial X$ (resp. $[0, a]\times_{\sin t}\partial X$).
\end{prop}

As one can see in our proof of \autoref{thm:b:k=1} and \autoref{thm:b:k=0}, the same type of rigidity holds, that is.
	
\begin{prop}[$\circled{2}=\circled{5}$]\label{prop:rigidty25}
The equality $\rad(X)=\frac1A$ in \eqref{eq:equality:0} (resp. $\rad(X)=\arccot(A)$ in \eqref{eq:equality:1}) implies that the space $X$ is isometric to the cone $[0, a]\times_{t} \partial X$ (resp. $[0, a]\times_{\sin t}\partial X$).
\end{prop}

The following example shows that class of spaces satisfying $\rad(X)=b=\sqrt{2\frac aA-a^{2}}$ or $\rad(X)=b=\arccos\left( \frac{A}{A\cos a+\sin a}\right)$ are very large.
\begin{example}[\circled{2}=\circled{4}]\label{ex:cutthetip}
We construct $X\in\alex^{3} (0)$, and $\BA(\partial X)\ge A\geq 0$ in the Euclidean space $\RR^3$ as the intersection of three balls centered at $(\frac{1}{A}-a,0,0),$ $-(\frac{1}{A}-a,0,0),$ and $(0,\sqrt{2\frac{a}{A}-a^2}-\epsilon-\frac{1}{A},0)$ respectively, where $a<\frac{1}{A}$, $\epsilon<\sqrt{2\frac{a}{A}-a^2}-a$. Soul of $X$ is the origin of $\RR^3$, the inner radius of $X$ is $a$ while the radius is $\sqrt{2\frac{a}{A}-a^2}$, which is also the distance from the soul to the boundary of $X$. A similar example in $\alex^{3}(1)$ can be constructed easily.
\end{example}

\begin{prop}[\circled{1}=\circled{4}]\label{prop:rigidty14}
The equality $a=\sqrt{2\frac aA-a^{2}}$ in \eqref{eq:equality:0} (resp. $a=\arccot(A)$ in \eqref{eq:equality:1}) follows the equivalent in \ref{prop:rigidty15}, thus that the space $X$ is isometric to the cone $[0, a]\times_{t} \partial X$ (resp. $[0, a]\times_{\sin t}\partial X$).
\end{prop}

The case $\kappa=1, A=0$ contains all positively curved Alexandrov spaces with boundary. The upper bound \circled{5} in \eqref{eq:equality:1} is $\pi/2$. In this case, the following rigidity theorem is proved by Petersen and Grove
\begin{prop}[\cite{GP2018}]\label{prop:PositiveRigidity}
Let $X\in \alex^{n}(1)$ and $\partial X$ is intrinsically isometric to $\SS^{n-1}$. Then $X$ is isometric to the lens $L_{\alpha}^{n}=[0, \alpha]*\SS^{n-2}$ for some $0<\alpha\le \pi$, where $*$ is the spherical join.
\end{prop}

\section{The Filling of Round Sphere}
In this section, we prove \autoref{thm:fill01}. The volume estimate uses the same idea as Petrunin's in \cite{Pet2007} we include it only for completeness. The rigidity part uses our discussion on the equality case in the previous section.
\begin{proof}[Proof of \autoref{thm:fill01}]
For $X\in \alex^{n}(0)$ with no empty boundary, the distance function to the boundary is concave in $X$. Thus the gradient exponential map $\gexp_s$ maps $\overline B_b(o_s)$ onto $X$. Moreover $\gexp_s$ also gives a homotopy equivalence of $\partial B_b(o_s)=\Sigma_{s}$ and $X\setminus \{s\}$, which is homotopy to $\partial X$, by noting that the soul $s$ is the only critical point of the distance function to $\partial X$. Since $\Sigma_{s}$ is a compact Alexandrov space without boundary, we have $H_{n-1}(\bx,\ZZ_2) \neq 0$. Hence $\forall x\in \bx$, the geodesic $sx$ must have a point of $\gexp_s\big(\partial\overline B_b(o_s)\big)$. Since the inverse of the gradient exponential map $\gexp_s^{-1}$ is uniquely defined inside any geodesic starting at $s$, it can only be $x$ for $\gexp_s$ is a short map. Thus
$$
\bx\subset \gexp_s(\partial B_b(o_s)).
$$
On another hand, Using gradient exponential map is a distance non-increasing map, we have
\begin{align*}
\vol_{n-1}(\bx)&\leq \vol_{n-1}\big(\gexp_s(\partial B_b(o_s)) \big)\\
&\leq \vol_{n-1}\big(\partial B_b(o_s)\big)         \\
&\leq \vol_{n-1}\big(\SS^{n-1}(b)\big)              \\
&\leq \vol_{n-1}\big(\SS^{n-1}(1)\big).
\end{align*}
If $\partial X$ is intrinsically isometric to $\SS^{n-1}$, the previous inequalities implies $b=1$.
Recall
$$b\leq\sqrt{2a-a^2}\leq 1,$$
we get $a=1$ thus by the Corollary 1.10 in \cite{AB2010}, $X$ is isometric to the ball of radius $1$ about the vertex in a $0$-cone over it's boundary. Since $\partial X$ is isometric to $\SS^{n-1}$,  such cone is $\RR^{n}$, therefore the conclusion holds.
\end{proof}
	
\bibliographystyle{alpha}
\bibliography{mybib}

\end{document}